\documentclass[12pt]{amsart}

\usepackage{tikz}

\newcommand{\elle}[1]{L^{#1}(\Omega)}
\newcommand{\hm}{H^{-1}(\Omega)}
\newcommand{\huz}{H^{1}_{0}(\Omega)}
\newcommand{\sob}[2]{W^{#1}_{#2}(\Omega)}

\newcommand{\erre}{\mathbb{R}}

\newcommand{\io}{\int_{\Omega}}
\newcommand{\ik}{\int_{\{|\un| \geq k\}}}
\newcommand{\ilk}{\int_{\{ |\un|< k\}}}
\newcommand{\ie}{\int_{E}}

\newcommand{\car}[1]{\raise2pt\hbox{$\chi$}_{#1}}
\newcommand{\dive}{{\rm div}}

\newcommand{\fn}{f_{n}}
\newcommand{\un}{u_{n}}
\newcommand{\zn}{z_{n}}

\newcommand{\vp}{\varphi}
\newcommand{\eps}{\varepsilon}

\newcommand{\arrstre}{\renewcommand{\arraystretch}{1.8}}
\newcommand{\norma}[2]{\|#1\|_{\lower 4pt \hbox{$\scriptstyle #2$}}}

\newcommand{\disp}{\displaystyle}
\newcommand{\ba}{\begin{array}}
\newcommand{\ea}{\end{array}}
\newcommand{\be}{\begin{equation}}
\newcommand{\ee}{\end{equation}}

\numberwithin{equation}{section}

\newcommand{\rife}[1]{(\ref{#1})}
\newtheorem{example}{\sc Example}[section]
\newtheorem{lemma}[example]{\sc Lemma}
\newtheorem{ohss}[example]{\sc Remark}
\newtheorem{prop}[example]{\sc Proposition}
\newtheorem{cor}[example]{\sc Corollary}
\newtheorem{theo}[example]{\sc Theorem}

\begin{document}

\title{Nonlinear degenerate elliptic problems with $\sob{1,1}{0}$ solutions}
\author{Lucio Boccardo, Gisella Croce, Luigi Orsina}
\address{L.B. -- Dipartimento di Matematica, ``Sapienza'' Universit\`{a} di Roma,
P.le A. Moro 2, 00185 Roma (ITALY)}
\email{boccardo@mat.uniroma1.it}
\address{G.C. -- Laboratoire de Math\'ematiques Appliqu\'ees du Havre, Universit\'e du Havre,
25, rue Philippe Lebon, 76063 Le Havre (FRANCE)}
\email{gisella.croce@univ-lehavre.fr}
\address{L.O. -- Dipartimento di Matematica, ``Sapienza'' Universit\`{a} di Roma,
P.le A. Moro 2, 00185 Roma (ITALY)}
\email{orsina@mat.uniroma1.it}

\begin{abstract}
We study a nonlinear equation with an elliptic operator having degenerate coercivity. We prove the existence of a unique $\sob{1,1}{0}$ distributional solution under suitable summability assumptions on the source in Lebesgue spaces. Moreover, we prove that our problem has no solution if the source is a Radon measure concentrated on a set of zero harmonic capacity.
\end{abstract}

\maketitle

\section{Introduction and statement of the results}

In this paper we are going to study the nonlinear elliptic equation
\be\label{P}
\left\{
\arrstre
\ba{cl}
\disp
-\dive\bigg(\frac{a(x,\nabla u)}{(1+|u|)^\gamma}\bigg) + u = f & \mbox{in $\Omega $,}\\
\hfill u = 0 \hfill & \mbox{on $\partial\Omega $,}
\ea
\right.
\ee
under the following assumptions.
The set $\Omega$ is a bounded, open subset of $\erre^{N}$, with $N > 2$, $\gamma > 0$, $f$ belongs to some Lebesgue space, and $a: \Omega \times \erre^{N} \to \erre^{N}$ is a Carath\'{e}odory function (i.e., $a(\cdot,\xi)$ is measurable on $\Omega$ for every $\xi$ in $\erre^{N}$, and $a(x,\cdot)$ is continuous on $\erre^{N}$ for almost every $x$ in $\Omega$) such that
\be\label{ell}
a(x,\xi) \cdot \xi \geq \alpha \,|\xi|^{2}\,,
\ee
\be\label{bdd}
|a(x,\xi)| \leq \beta\,|\xi|\,,
\ee
\be\label{monot}
[a(x,\xi) - a(x,\eta)]\cdot(\xi - \eta) > 0\,,
\ee
for almost every $x$ in $\Omega$ and for every $\xi$ and $\eta$ in $\erre^{N}$, $\xi \neq \eta$, where $\alpha$ and $\beta$ are positive constants. We are going to prove that, under suitable assumptions on $\gamma$ and $f$, problem \rife{P} has a unique distributional solution $u$ obtained by approximation, with $u$ belonging to the (nonreflexive) Sobolev space $\sob{1,1}{0}$. Furthermore, we are going to prove that problem \rife{P} does not have a solution if $\gamma > 1$ and the datum $f$ is a bounded Radon measure concentrated on a set of zero harmonic capacity.

Problems like \rife{P} have been extensively studied in the past. In \cite{bdo} (see also \cite{gp1}, \cite{gp2}, \cite{Po}), existence and regularity results were proved, under the assumption that $a(x,\xi) = A(x)\,\xi$, with $A$ a uniformly elliptic bounded matrix, and $0 < \gamma \leq 1$, for the problem
\be
\left\{
\arrstre
\ba{cl}
\disp
-\dive\bigg(\frac{A(x)\,\nabla u}{(1+|u|)^\gamma}\bigg) = f & \mbox{in $\Omega $,}\\
\hfill u = 0 \hfill & \mbox{on $\partial\Omega$,}
\ea
\right.
\label{p-teta}
\ee
where $f$ belongs to $\elle{m}$ for some $m \geq 1$.

The main difficulty in dealing with problem \rife{p-teta} (or \rife{P}) is that the differential operator, even if well defined between $\huz$ and its dual $\hm$, is not coercive on $\huz$ due to the fact that if $u$ is large, $\frac{1}{(1+|u|)^\gamma}$ tends to zero (see \cite{Po} for an explicit example).

This lack of coercivity implies that the classical methods used in order to prove the existence of a solution for elliptic equations (see \cite{LL}) cannot be applied even if the datum $f$ is regular. However, in \cite{bdo}, a whole range of existence results was proved, yielding solutions belonging to some Sobolev space $\sob{1,q}0$, with $q=q(\gamma,m)\leq 2$, if $f$ is regular enough. Under weaker summability assumptions on $f$, the gradient of $u$ (and even $u$ itself) may not be in $\elle 1$: in this case, it is possible to give a meaning to solutions of problem \rife{p-teta}, using the concept of {\sl entropy solutions\/} which has been introduced in \cite{BBGGPV}.

If $\gamma > 1$, a non existence result for problem \rife{p-teta} was proved in \cite{ABFOT} (where the principal part is nonlinear with respect to the gradient), even for $\elle\infty$ data $f$. Therefore, if the operator becomes ``too degenerate'', existence may be lost even for data expected to give bounded solutions. However, as proved in \cite{bb}, existence of solutions can be recovered by adding a lower order term of order zero. Indeed, if we consider the problem
\be
\left\{
\arrstre
\ba{cl}
\disp
-\dive\bigg(\frac{A(x)\nabla u}{(1+|u|)^{\gamma}}\bigg) + u = f& \mbox{in $\Omega $,}\\
\hfill u = 0 \hfill & \mbox{on $\partial\Omega$,}
\ea
\right.
\label{p-2}
\ee
with $f$ in $\elle{m}$, then the following results can be proved in the case $\gamma > 1$ (see \cite{bb} and \cite{croce}):

\begin{enumerate}
\item[i)]
if $ m > \gamma\frac{N}{2}$, then there
exists a weak solution in $\huz \cap \elle\infty$;
\item[ii)]
if $m\geq \gamma + 2$, then there exists a weak solution in $\huz \cap \elle{m}$;
\item[iii)]
if $\frac{\gamma + 2}{2} < m < \gamma + 2$, then there exists a distributional solution in $ \sob{1,\frac{2m}{\gamma + 2}}{0} \cap \elle{m}$;
\item[iv)]
if $1 \leq	m \leq \frac{\gamma + 2}{2}$, then there exists an entropy solution in $\elle{m}$ whose gradient belongs to the Marcinkiewicz space $M^{\frac{2m}{\gamma + 2}}(\Omega)$.
\end{enumerate}

Note that if $\gamma + 2 \leq m < \gamma\frac{N}{2}$ and $m$ tends to $\gamma\frac{N}{2}$, the summability result of ii) is not ``continuous'' with the boundedness result of i), according to the following example (see also Example 3.3 of \cite{bb}).

\begin{example}\label{noreg}\rm
If $\frac{2}{\gamma} < \sigma < N-2$, then $u(x) = \frac{1}{|x|^\sigma}-1$ is a distributional solution of \rife{p-2} with $A(x) \equiv I$, and $f(x) = \frac{\sigma(N-2+\sigma(\gamma-1))}{|x|^{2-\sigma(\gamma-1)}}+\frac{1}{|x|^\sigma}-1$. Due to the assumptions on $\sigma$, both $f$ and $u$ belong to $\elle{m}$, with $m < \gamma\frac{N}{2}$. If $m$ tends to $\gamma \frac{N}{2}$, i.e., if $\sigma$ tends to $\frac{2}{\gamma}$, the solution $u$ does not become bounded.
\end{example}

As stated before, this paper is concerned with two borderline cases connected with point iv) above:
\begin{enumerate}
\item[A.] if $m = \frac{\gamma+2}{2}$, we will prove in Section 2 the existence of $\sob{1,1}{0}$ distributional solutions, and in Section 3 their uniqueness;
\item[B.] if $f$ is a bounded Radon measure concentrated on a set $E$ of zero harmonic capacity and $\gamma > 1$, we will prove in Section 4 non existence of solutions.
\end{enumerate}

In the linear case, i.e., for the boundary value problem \rife{p-2}, a simple proof of the existence result is given in \cite{bco}.

\begin{ohss}\rm
Let $a(x,\xi) = A(x)\,\xi$, with $A$ a bounded and measurable uniformly elliptic matrix, and let $u \geq 0$ be a solution of
$$
-\dive\bigg(\frac{A(x)\nabla u}{(1+u)^{\gamma}}\bigg) + u = f\,,
$$
with $\gamma > 1$ and $f \geq 0$. If we define
$$
z = \frac{1}{\gamma-1}\bigg(1-\frac{1}{(1+u)^{\gamma-1}}\bigg)\,,
$$
then $z$ is a solution of
$$
-\dive (A(x)\nabla z ) + \bigg(\frac{1}{(1-(\gamma-1)z)^{\frac{1}{\gamma-1}}}-1\bigg) = f\,,
$$
which is an equation whose lower order term becomes singular as $z$ tends to the value $\frac{1}{\gamma-1}$. For a study of these problems, see \cite{b} and \cite{dpp}.
\end{ohss}
\begin{ohss}\rm
We explicitely state that our existence results can be generalized to equations with differential operators defined on $\sob{1,p}{0}$, with $p > 1$: if $\gamma\geq\frac{(p-2)_+}{p-1}$ and if $m = \frac{\gamma(p-1)+2}{p}$, then it is possible to prove the existence of a distributional solution $u$ in $W_0^{1,1}(\Omega) \cap \elle{m}$ of the boundary value problem
\be
\label{pp}
\left\{
\arrstre
\ba{cl}
\disp
-\dive\bigg(\frac{a(x,\nabla u)}{(1+|u|)^{\gamma(p-1)}}\bigg) + u = f & \mbox{in $\Omega $,}\\
\hfill u = 0 \hfill & \mbox{on $\partial\Omega$,}
\ea
\right.
\ee
where $a(x,\xi)$ satisfies \rife{ell}, \rife{bdd} and \rife{monot} with $p$ instead of~2 (in \rife{bdd}, $a$ grows as $|\xi|^{p-1}$).
\end{ohss}

\section{Existence of a $\sob{1,1}{0}$ solution}

In this section we prove the existence of a $\sob{1,1}{0}$ solution to problem \rife{P}. Our result is the following.

\begin{theo}\label{teoroma}\sl
Let $\gamma > 0$, and let $f$ be a function in $L^{\frac{\gamma+2}{2}}(\Omega)$. Then there exists a distributional solution $u$ in $W_0^{1,1}(\Omega )\cap L^{\frac{\gamma+2}{2}}(\Omega)$ of \rife{P}, that is,
\be\label{eq-distr}
\io\frac{a(x,\nabla u) \cdot \nabla \vp}{(1+|u|)^{\gamma}}
+
\io u\,\vp=\io f\,\vp\,,
\quad
\forall\vp\in \sob{1,\infty}{0}\,.
\ee
\end{theo}

\begin{ohss}\rm
The previous result gives existence of a solution $u$ in $\sob{1,1}{0}$ to \rife{p-2} for every $\gamma > 0$.
If $0 < \gamma \leq 1$ existence results for \rife{P} can also be proved by the same techniques of \cite{bdo}. More precisely, if $f$ belongs to $\elle m$ with $m > \frac{N}{N(1-\gamma)+1+\gamma}$ then \rife{P} has a solution in $\sob{1,q}{0}$, with $q = \frac{Nm(1-\gamma)}{N-m(1+\gamma)}$. Note that when $m$ tends to $\frac{N}{N(1-\gamma)+1+\gamma}$, then $q$ tends to 1. We have now two cases: if $\frac{\gamma+2}{2} > \frac{N}{N(1-\gamma)+1+\gamma}$, that is, if $0 < \gamma < \frac{2}{N-1}$, our result is weaker than the one in \cite{bdo}.
On the other hand, if $\frac{2}{N-1} \leq \gamma \leq 1$, then our result, which strongly uses the lower order term of order zero, is better.
\end{ohss}
\begin{ohss}\rm
The same existence result, with the same proof, holds for the following
boundary value problem
$$
\left\{
\ba{cl}
\disp
-\dive (b(x,u,\nabla u)) + u = f & \mbox{in $\Omega $,}\\
\hfill u = 0 \hfill & \mbox{on $\partial\Omega$,}
\ea
\right.
$$
with $b : \Omega \times \erre \times \erre^{N} \to \erre^{N}$ a Carath\'{e}odory function such that
$$
\frac{\alpha|\xi|^2}{(1+|s|)^{\gamma}}\leq b(x,s,\xi)\cdot\xi
 \leq\beta|\xi|^2\,,
$$
where $\alpha$, $\beta$, $\gamma$ are positive constants.
\end{ohss}

To prove Theorem \ref{teoroma} we will work by approximation. First of all, let $g$ be a function in $\elle\infty$. Then, by the results of \cite{bb}, there exists a solution $v$ in $\huz \cap \elle\infty$ of
\be\label{pinfty}
\left\{
\arrstre
\ba{cl}
\disp
-\dive\bigg(\frac{a(x,\nabla v)}{(1+|v|)^{\gamma}}\bigg) + v = g & \mbox{in $\Omega$,} \\
\hfill v = 0 \hfill & \mbox{on $\partial\Omega$.}
\ea
\right.
\ee
In order for this paper to be self contained, we give here the easy proof of this fact. Let $M = \norma{g}{\elle\infty} + 1$, and consider the problem
\be\label{pbbase}
\left\{
\arrstre
\ba{cl}
\disp
-\dive\bigg(\frac{a(x,\nabla v)}{(1+|T_{M}(v)|)^{\gamma}}\bigg) + v = g & \mbox{in $\Omega$,} \\
\hfill v = 0 \hfill & \mbox{on $\partial\Omega$.}
\ea
\right.
\ee
Here and in the following we define $T_{k}(s) = \max(-k,\min(s,k))$ for $k \geq 0$ and $s$ in $\erre$. Since the differential operator is pseudomonotone and coercive thanks to the assumptions on $a$ and to the truncature, by the results of \cite{LL} there exists a weak solution $v$ in $\huz$ of \rife{pbbase}. Choosing $(|v| - \norma{g}{\elle\infty})_{+}\,{\rm sgn}(v)$ as a test function we obtain, dropping the nonnegative first term,
$$
\io |v|\,(|v| - \norma{g}{\elle\infty})_{+} \leq \io \norma{g}{\elle\infty}\,(|v| - \norma{g}{\elle\infty})_{+}\,.
$$
Thus,
$$
\io (|v| - \norma{g}{\elle\infty})\,(|v| - \norma{g}{\elle\infty})_{+} \leq 0\,,
$$
so that $|v| \leq \norma{g}{\elle\infty} < M$. Therefore, $T_{M}(v) = v$, and $v$ is a bounded weak solution of \rife{pinfty}.

Let now $\fn$ be a sequence of $\elle\infty$ functions which converges to $f$ in $\elle{\frac{\gamma+2}{2}}$, and such that $|\fn| \leq |f|$ almost everywhere in $\Omega$, and consider the approximating problems
\be
\left\{
\arrstre
\ba{cl}
\disp
-\dive\bigg(\frac{a(x,\nabla \un)}{(1+|\un|)^{\gamma}}\bigg) + \un = \fn & \mbox{in $\Omega$,} \\
\hfill \un = 0 \hfill & \mbox{on $\partial\Omega$.}
\ea
\right.
\label{ppn_0}
\ee
A solution $\un$ in $\huz \cap \elle\infty$ exists choosing $g = \fn$ in \rife{pinfty}. We begin with some {\sl a priori\/} estimates on the sequence $\{\un\}$.

\begin{lemma}\label{lemma_stime}\sl
If $\un$ is a solution to problem \rife{ppn_0}, then, for every $k\geq0$,
\be\label{aa}
\ik|\un|^{\frac{\gamma+2}{2}} \leq \ik|f|^{\frac{\gamma+2}{2}}\,;
\ee

\be\label{bb}
\ik\frac{|\nabla \un|^2}{(1+| \un |)^{\frac{\gamma+2}{2}}}
\leq
C\,
\bigg[\ik |f|^{\frac{\gamma+2}{2}}\bigg]^{\frac{2}{\gamma+2}}\,;
\ee

\be\label{cc}
\ik |\nabla \un|
\leq
C
\bigg[
\ik |f|^\frac{\gamma+2}{2}
\bigg]^{\frac{1}{\gamma+2}}\,;
\ee

\be\label{dd}
\alpha\io|\nabla T_k(\un)|^2 \leq k\,(1 + k)^{\gamma}\io|f|\,.
\ee
Here, and in the following, $C$ denotes a positive constant depending on $\alpha$, $\gamma$, ${\rm meas}(\Omega)$, and the norm of $f$ in $\elle{\frac{\gamma+2}{2}}$.
\end{lemma}

\begin{proof}
Let $k \geq 0$, $h > 0$, and let $\psi_{h,k}(s)$ be the function defined by
$$
\psi_{h,k}(s) =
\left\{
\ba{cl}
0 & \mbox{if $0 \leq s \leq k$,}
\\
h\,(s-k) & \mbox{if $k < s \leq k + \frac1h$,}
\\
1& \mbox{if $s > k + \frac1h$,}
\\
\psi_{h,k}(s)=-\psi_{h,k}(-s) & \mbox{if $s < 0$.}
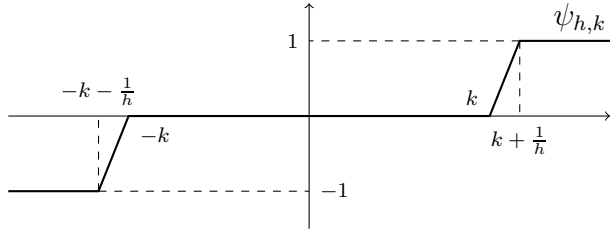
\ea
\right.
$$
\begin{center}
\begin{tikzpicture}
	\draw[->] (-4.0, 0.0) -- ( 4.0, 0.0);
	\draw[->] ( 0.0,-1.5) -- ( 0.0, 1.5);
	\draw[thick] (-4.0,-1.0) -- (-2.8,-1.0) -- (-2.4, 0.0) node[anchor=north west]{\tiny$-k$} -- ( 2.4, 0.0) node[anchor=south east]{\tiny$k$} -- ( 2.8, 1.0) -- ( 4.0, 1.0);
	\draw ( 3.6, 1.3) node{$\psi_{h,k}$};
	\draw[dashed] (-2.8,-1.0) -- ( 0.0,-1.0) node[right]{\tiny$-1$};
	\draw[dashed] ( 2.8, 1.0) -- ( 0.0, 1.0) node[left]{\tiny$ 1$};
	\draw[dashed] (-2.8,-1.0) -- (-2.8, 0.0) node[above]{\tiny$-k-\frac{1}{h}$};
	\draw[dashed] ( 2.8, 1.0) -- ( 2.8, 0.0) node[below]{\tiny$k+\frac{1}{h}$};
\end{tikzpicture}
\end{center}

Note that
$$
\lim_{h \to +\infty}\,\psi_{h,k}(s) = \left\{
\ba{cl}
1& \mbox{if $s > k$,}
\\
0 & \mbox{if $|s| \leq k$,}
\\
-1 & \mbox{if $s < -k$.}
\ea
\right.
$$
Let $\eps > 0$, and choose $(\eps + |\un|)^{\frac{\gamma}{2}}\,\psi_{h,k}(\un)$ as a test function in \rife{ppn_0}; such a test function is admissible since $\un$ belongs to $\huz \cap \elle\infty$ and $\psi_{h,k}(0) = 0$. We obtain
\be\label{estimina}
\arrstre
\ba{l}
\disp
\frac{\gamma}{2}
\io a(x,\nabla\un) \cdot \nabla\un\frac{(\eps + |\un|)^{\frac{\gamma}{2}-1}}{(1 + |\un|)^{\gamma}}\,|\psi_{h,k}(\un)|
\\
\qquad
\disp
+
\io \frac{a(x,\nabla\un) \cdot \nabla\un}{(1 + |\un|)^{\gamma}}\,\psi'_{h,k}(\un)\,(\eps + |\un|)^{\frac{\gamma}{2}}
\\
\qquad
\disp
+
\io \un \,(\eps + |\un|)^{\frac{\gamma}{2}}\,\psi_{h,k}(\un)
\\
\quad
\disp
=
\io \fn\,(\eps + |\un|)^{\frac{\gamma}{2}}\,\psi_{h,k}(\un)
\,.
\ea
\ee
By \rife{ell}, and since $\psi'_{h,k}(s) \geq 0$, the first two terms are nonnegative, so that we obtain, recalling that $|\fn| \leq |f|$,
$$
\io \un \, (\eps + |\un|)^{\frac{\gamma}{2}}\,\psi_{h,k}(\un)
\leq
\io |f|\,(\eps + |\un|)^{\frac{\gamma}{2}}\,|\psi_{h,k}(\un)|
\,.
$$
Letting $\eps$ tend to zero and $h$ tend to infinity, we obtain, by Fatou's lemma (on the left hand side) and by Lebesgue's theorem (on the right hand side, recall that $\un$ belongs to $\elle\infty$),
$$
\ik |\un|^{\frac{\gamma+2}{2}}
\leq
\ik |f|\,|\un|^{\frac{\gamma}{2}}
\,.
$$
Using H\"older's inequality on the right hand side we obtain
$$
\ik|\un|^{\frac{\gamma+2}{2}}
\leq
\bigg[\ik |f|^{\frac{\gamma+2}{2}}\bigg]^{\frac{2}{\gamma+2}}
\,
\bigg[\ik |\un|^{\frac{\gamma+2}{2}} \bigg]^{\frac{\gamma}{\gamma+2}}
\,.
$$
Simplifying equal terms we thus have
$$
\ik|\un|^\frac{\gamma+2}{2} \leq \ik |f|^\frac{\gamma+2}{2}\,,
$$
which is \rife{aa}. Note that from \rife{aa}, written for $k = 0$, it follows
\be\label{stimanorma}
\io |\un|^{\frac{\gamma+2}{2}} \leq \io |f|^{\frac{\gamma+2}{2}} = \norma{f}{\elle{\frac{\gamma+2}{2}}}^{\frac{\gamma+2}{2}}\,.
\ee

Now we consider \rife{estimina} written for $\eps = 1$. Dropping the nonnegative second and third terms, and using that $|\fn| \leq |f|$, we have
$$
\frac{\gamma}{2}\io \frac{a(x,\nabla\un) \cdot \nabla\un}{(1 + |\un|)^{\frac{\gamma+2}{2}}}\,|\psi_{h,k}(\un)|
\leq
\io |f|(1 + |\un|)^{\frac{\gamma+2}{2}}|\psi_{h,k}(\un)|\,.
$$
Using \rife{ell}, and letting $h$ tend to infinity, we get (using again Fatou's lemma and Lebesgue's theorem)
$$
\alpha\frac{\gamma}{2}\ik \frac{|\nabla\un|^{2}}{(1 + |\un|)^{\frac{\gamma+2}{2}}}
\leq
\ik |f|(1 + |\un|)^{\frac{\gamma}{2}}\,.
$$
H\"older's inequality on the right hand side then gives
$$
\arrstre
\ba{l}
\disp
\alpha\frac{\gamma}{2}
\ik \frac{|\nabla \un|^2}{(1+| \un |)^{\frac{\gamma+2}{2}}}
\\
\qquad
\disp
\leq
\bigg[\ik |f|^{\frac{\gamma+2}{2}} \bigg]^{\frac{2}{\gamma + 2}}
\bigg[\ik (1 + |\un|)^{\frac{\gamma+2}{2}} \bigg]^{\frac{\gamma}{\gamma + 2}}
\\
\qquad
\disp
\leq
\bigg[\ik |f|^{\frac{\gamma+2}{2}} \bigg]^{\frac{2}{\gamma + 2}}
\bigg[\io (1 + |\un|)^{\frac{\gamma+2}{2}} \bigg]^{\frac{\gamma}{\gamma + 2}}\,,
\ea
$$
so that, by \rife{stimanorma},
$$
\alpha\frac{\gamma}{2}
\ik \frac{|\nabla \un|^2}{(1+| \un |)^{\frac{\gamma+2}{2}}}
\leq
C\,\bigg[\ik |f|^{\frac{\gamma+2}{2}}\bigg]^{\frac{2}{\gamma+2}}\,,
$$
which is \rife{bb}.

Then, again by H\"{o}lder's inequality, and by \rife{bb} and \rife{stimanorma},
\be\label{conlog}
\arrstre
\ba{l}
\disp
\ik|\nabla \un|
=
\ik \frac{|\nabla \un|}{(1 + |\un|)^{\frac{\gamma + 2}{4}}}\,(1 + |\un|)^{\frac{\gamma + 2}{4}}
\\
\qquad
\disp
\leq
\bigg[
\ik\frac{|\nabla \un|^2}{(1+| \un |)^\frac{\gamma+2}{2}}
\bigg]^\frac 12
\bigg[\ik (1+|\un|)^{\frac{\gamma+2}{2}}\bigg]^\frac{1}{2}
\\
\qquad
\disp
\leq
C\,
\bigg[\ik |f|^\frac{\gamma+2}{2}\bigg]^{\frac{1}{\gamma+2}}
\bigg[\io (1+|\un|)^{\frac{\gamma+2}{2}}\bigg]^\frac{1}{2}
\\
\qquad
\disp
\leq
C\,
\bigg[\ik |f|^\frac{\gamma+2}{2}\bigg]^{\frac{1}{\gamma+2}}\,,
\ea
\ee
so that \rife{cc} is proved.

Finally, choosing $T_k(\un)$ as a test function in \rife{ppn_0} we get, dropping the nonnegative linear term, and using \rife{ell},
$$
\alpha\io|\nabla T_k(\un)|^2\leq k(1+k)^{\gamma}\io|f|\,,
$$
which is \rife{dd}.
\end{proof}

\begin{lemma}\label{uqo}\sl
If $\{\un\}$ is the sequence of solutions to \rife{ppn_0},
there exists a subsequence, still denoted by $\{\un\}$, and a function $u$ in $\elle{\frac{\gamma+2}{2}}$, with $T_k(u)$ belonging to $\huz$ for every $k > 0$, such that $\un$ almost everywhere converges to $u$ in $\Omega$, and $T_k(\un)$ weakly converges to $T_k(u)$ in $\huz$.
\end{lemma}

\begin{proof}
Consider \rife{bb} written for $k = 0$:
\be\label{bb1}
\io \frac{|\nabla\un|^{2}}{(1 + |\un|)^{\frac{\gamma+2}{2}}}
\leq
C\,\norma{f}{\elle{\frac{\gamma+2}{2}}}\,.
\ee
Since (if $\gamma \neq 2$)
$$
\frac{|\nabla\un|^{2}}{(1 + |\un|)^{\frac{\gamma+2}{2}}}
=
\frac{16}{(2 - \gamma)^{2}}\,|\nabla [(1 + |\un|)^{\frac{2-\gamma}{4}} - 1]|^{2}\,,
$$
the sequence $v_{n} = \frac{4}{2-\gamma}[(1 + |\un|)^{\frac{2-\gamma}{4}} - 1]{\rm sgn}(\un)$ is bounded in $\huz$ by \rife{bb1}. If $\gamma = 2$ we have
$$
\frac{|\nabla\un|^{2}}{(1 + |\un|)^{2}} = |\nabla \log(1 + |\un|)|^{2}\,,
$$
so that $v_{n} = [\log(1 + |\un|)]{\rm sgn}(\un)$ is bounded in $\huz$. In both cases, up to a subsequence still denoted by $v_{n}$, $v_{n}$ converges to some function $v$ weakly in $\huz$, strongly in $\elle2$, and almost everywhere in $\Omega$. If $\gamma < 2$, define
$$
u(x) = \bigg[\bigg( \frac{2-\gamma}{4}|v(x)| + 1\bigg)^{\frac{4}{2-\gamma}} - 1\bigg]{\rm sgn}(v(x))\,,
$$
if $\gamma > 2$ define
$$
u(x) =
\left\{
\ba{cl}
\bigg[\bigg( \frac{2-\gamma}{4}|v(x)| + 1\bigg)^{\frac{4}{2-\gamma}} - 1\bigg]{\rm sgn}(v(x)) & \mbox{if $|v(x)| < \frac{4}{\gamma-2}$,}
\\
+\infty & \mbox{if $v(x) = \frac{4}{\gamma-2}$,}
\\
-\infty & \mbox{if $v(x) = -\frac{4}{\gamma-2}$,}
\ea
\right.
$$
while if $\gamma = 2$, define
$$
u(x) = [{\rm e}^{|v(x)|} - 1]{\rm sgn}(v(x))\,.
$$
Thus, $\un$ almost everywhere converges, up to a subsequence still denoted by $\un$, to $u$. From now on, we will consider this particular subsequence, for which it holds that $\un$ almost everywhere converges to $u$.

We use now \rife{aa} written for $k = 0$:
$$
\io |\un|^{\frac{\gamma+2}{2}} \leq \io |f|^{\frac{\gamma+2}{2}} \leq C\,.
$$
Since $\un$ almost everywhere converges to $u$, we have from Fatou's lemma that
$$
\io |u|^{\frac{\gamma+2}{2}} \leq C\,.
$$
Hence $u$ belongs to $\elle{\frac{\gamma+2}{2}}$, which implies that $u$ is almost everywhere finite (note that if $\gamma > 2$ this fact did not follow from the definition of $u$, since $|v|$ could have assumed the value $\frac{4}{\gamma-2}$ on a set of positive measure).

Let now $k > 0$; since from \rife{dd} it follows that the sequence $\{T_{k}(\un)\}$ is bounded in $\huz$, there exists a subsequence $T_{k}(u_{n_{j}})$ which weakly converges to some function $v_{k}$ in $\huz$. Using the almost everywhere convergence of $\un$ to $u$, we have that $v_{k} = T_{k}(u)$. Since the limit is independent on the subsequence, then the whole sequence $\{T_{k}(\un)\}$ weakly converges to $T_{k}(u)$, for every $k > 0$.
\end{proof}

\begin{ohss}\label{grad}\rm
Using the fact that $T_{k}(u)$ is in $\huz$ for every $k > 0$, and the results of \cite{BBGGPV}, we have that there exists a unique measurable function $v$ with values in $\erre^{N}$, such that
$$
\nabla T_{k}(u) = v\,\car{\{|u| \leq k\}}
\qquad
\mbox{almost everywhere in $\Omega$, for every $k > 0$.}
$$
Following again \cite{BBGGPV}, we will {\sl define} $\nabla u = v$, the approximate gradient of $u$.
\end{ohss}

\begin{ohss}\label{interpo}\rm
We emphasize that if $\gamma = 2$, then \rife{conlog}, written for $k = 0$, becomes
$$
\io |\nabla\un| \leq \bigg[ \io\frac{|\nabla\un|^{2}}{(1 + |\un|)^{2}}\bigg]^{\frac12} \bigg[ \io (1 + |\un|)^{2}\bigg]^{\frac12}\,.
$$
Since
$$
\frac{|\nabla\un|^{2}}{(1 + |\un|)^{2}} = |\nabla \log(1 + |\un|)|^{2}\,,
$$
a nonlinear interpolation result follows: let $A$ be in $\erre^+$ and let $v$ in $\elle2$ be such that $\log(A+|v|)$ belongs to $\huz$. Then $v$ belongs to $\sob{1,1}{0}$, and
$$
\io|\nabla v| \leq \norma{\log(A+|v|)}{\huz} \bigg[\io (A+|v|)^2\bigg]^\frac{1}{2}\,.
$$
\end{ohss}

Our next result deals with the strong convergence of $T_{k}(\un)$ in $\huz$.

\begin{prop}\label{leopo}\sl
Let $\un$ and $u$ be the sequence of solutions to problems \rife{ppn_0} and the function in $\elle{\frac{\gamma+2}{2}}$ given by Lemma \ref{uqo}.
Then, for every fixed $k>0$, $T_k(\un)$ strongly converges to $T_k(u)$ in $\huz$, as $n$ tends to infinity.
\end{prop}

\begin{proof}
We follow the proof of \cite{porretta-leone}.

Let $h > k$ and choose $T_{2k}[\un-T_h(\un) + T_k(\un)-T_k(u)]$ as a test function in \rife{ppn_0}. We have
\be
\arrstre
\ba{l}
\disp \io\frac{a(x,\nabla \un)\cdot
\nabla T_{2k}[\un-T_h(\un) + T_k(\un)-T_k(u)]}
{(1+|\un|)^{\gamma}}
\\
\qquad
\disp =-
\io
(\un-\fn )
\, T_{2k}[\un-T_h(\un) + T_k(\un)-T_k(u)]\,.
\ea
\label{first}
\ee
We observe that the right hand side converges to zero as first $n$ and then $h$ tend to infinity, since $\un$ converges to $u$ almost everywhere in $\Omega$ and $\un$ and $\fn$ are bounded in $L^{\frac{\gamma+2}{2}}(\Omega)$. Thus, if we define $\eps(n,h)$ as any quantity such that
$$
\lim_{h \to +\infty}\,\lim_{n \to +\infty}\,\eps(n,h) = 0\,,
$$
then$$
\io
(\un-\fn )
\, T_{2k}[\un-T_h(\un) + T_k(\un)-T_k(u)] = \eps(n,h)\,.
$$
Let $M=4k+h$. Observing that $\nabla T_{2k}[\un-T_h(\un) + T_k(\un)-T_k(u)] = 0$ if $|\un| \geq M$, by \rife{first} we have
$$
\arrstre
\ba{l}
\disp
\eps(n,h)
=
\ilk
\frac{a(x,\nabla T_M(\un))\cdot
\nabla [\un-T_h(\un) + T_k(\un)-T_k(u)]}
{(1+|\un|)^{\gamma}}
\\
\quad
\disp
+\ik\frac{a(x,\nabla T_M(\un))\cdot
\nabla [\un-T_h(\un) + T_k(\un)-T_k(u)]}
{(1+|\un|)^{\gamma}} \,.
\ea
$$
Since $\un-T_h(\un)=0$ in $\{|\un|\leq k\}$ and $\nabla T_k(\un)=0$ in $\{|\un|\geq k\}$,
we have, using that $a(x,0) = 0$,
$$
\arrstre
\ba{r@{\hspace{2pt}}c@{\hspace{2pt}}l}
\disp
\eps(n,h)
& = &
\disp
\io\frac{a(x,\nabla T_{k}(\un))\cdot\nabla [T_k(\un)-T_k(u)]}
{(1+|\un|)^{\gamma}}
\\
& &
\quad
\disp
+\ik\frac{a(x,\nabla T_M(\un))\cdot\nabla [\un-T_h(\un)]}
{(1+|\un|)^{\gamma}}
\\
& &
\quad
\disp
-
\ik\frac{a(x,\nabla T_M(\un))\cdot \nabla T_k(u)}{(1+|\un|)^{\gamma}}\,.
\ea
$$
The second term of the right hand side is positive, so that
$$
\arrstre
\ba{r@{\hspace{2pt}}c@{\hspace{2pt}}l}
\disp
\eps(n,h)
& \geq &
\disp
\io \frac{[a(x,\nabla T_k(\un))-a(x,\nabla T_k(u))] \cdot \nabla [T_k(\un)-T_k(u)]}{(1+k)^{\gamma}}
\\
& &
\quad
\disp
+
\io\frac{a(x,\nabla T_k(u)) \cdot \nabla [T_k(\un)-T_k(u)] }{(1+|\un|)^{\gamma}}
\\
& &
\quad
\disp
-
\ik\!\!\!\frac{a(x,\nabla T_M(\un)) \cdot \nabla T_k(u)}{(1+|\un|)^{\gamma}} = I_{n} + J_{n} - K_{n}\,.
\ea
$$
The last two terms tend to zero as $n$ tends to infinity. Indeed
$$
\lim_{n \to +\infty}\, J_{n} = \lim_{n \to +\infty}\io \frac{a(x,\nabla T_k(u)) \cdot \nabla [T_k(\un)-T_k(u)] }{(1+|\un|)^{\gamma}} = 0\,,
$$
since $T_k(\un)$ converges to $T_k(u)$ weakly in $\huz$ and $\frac{a(x,\nabla T_{k}(u))}{(1 + |\un|)^{\gamma}}$ is strongly compact in $(\elle2)^{N}$ by the growth assumption \rife{bdd} on $a$.
The last term can be rewritten as
$$
K_{n}
=
\io\frac{a(x,\nabla T_M(\un)) \cdot \nabla T_k(u)\chi_{\{|\un| \geq k\}}}{(1+|\un|)^{{\gamma}}}\,.
$$
Since $M$ is fixed with respect to $n$, then the sequence $\{a(x,\nabla T_{M}(\un))\}$ is bounded in $(\elle2)^{N}$. Hence, there exists $\sigma$ in $(\elle2)^{N}$, and a subsequence $\{a(x,\nabla T_{M}(u_{n_{j}}))\}$, such that
$$
\lim_{j \to +\infty}\,a(x,\nabla T_M(u_{n_{j}})) = \sigma\,,
$$
weakly in $(\elle2)^{N}$. On the other hand,
$$
\lim_{n \to +\infty}\,\frac{\nabla T_{k}(u) \chi_{\{k\leq|\un|\}}}{(1 + |\un|)^{\gamma}} = \frac{\nabla T_k(u)\chi_{\{k\leq|u|\}}}{(1 + |u|)^{\gamma}} = 0\,,
$$
strongly in $(\elle2)^{N}$, and so
$$
\lim_{j \to +\infty}\,K_{n_{j}}
=
\lim_{j \to +\infty} \int_{\{|u_{n_{j}}| \geq k\}}\!\!\!\frac{a(x,\nabla T_M(u_{n_{j}}))\cdot \nabla T_k(u)}{(1+|u_{n_{j}}|)^{\gamma}} = 0\,.
$$
Since the limit does not depend on the subsequence, we have
$$
\lim_{n \to +\infty}\,K_{n} = \lim_{n \to +\infty} \ik\!\!\!\frac{a(x,\nabla T_M(\un))\cdot \nabla T_k(u)}{(1+|\un|)^{\gamma}} = 0\,,
$$
as desired.
Therefore,
$$
\eps(n,h)
\geq
I_{n}
=
\io
\frac{[a(x,\nabla T_k(\un))-a(x,\nabla T_k(u))] \cdot \nabla [T_k(\un)-T_k(u)]}
{(1+k)^{\gamma}}
\,,
$$
so that, thanks to \rife{monot},
$$
\lim_{n \to +\infty}\,\io {[a(x,\nabla T_k(\un))-a(x,\nabla T_k(u))] \cdot \nabla [T_k(\un)-T_k(u)]} = 0\,.
$$
Using this formula, \rife{monot} and the results of \cite{Br} and \cite{BMP}, we then conclude that $T_{k}(\un)$ strongly converges to $T_{k}(u)$ in $\huz$, as desired.
\end{proof}

\begin{cor}\label{duqo}\sl
Let $\un$ and $u$ be as in Proposition \ref{leopo}. Then $\nabla \un$ converges to $\nabla u$ almost everywhere in $\Omega$, where $\nabla u$ has been defined in Remark \ref{grad}.
\end{cor}

\begin{lemma}\label{dul1}\sl
Let $\un$ and $u$ be as in Proposition \ref{leopo}. Then $\nabla \un$ strongly converges to $\nabla u$ in $(\elle1)^{N}$. Moreover $\un$ strongly converges to $u$ in $L^{\frac{\gamma+2}{2}}(\Omega)$.
\end{lemma}

\begin{proof}
We begin by proving the convergence of $\nabla \un$ to $\nabla u$. Let $\eps > 0$, and let $k > 0$ be sufficiently large such that
\be\label{sceltak}
\bigg[\ik |f|^{\frac{\gamma+2}{2}}\bigg]^{\frac{1}{\gamma+2}} < \eps\,,
\ee
uniformly with respect to $n$. This can be done thanks to \rife{stimanorma} and to the absolute continuity of the integral. Let $E$ be a measurable set. Writing
$$
\ie|\nabla \un| = \ie|\nabla T_{k}(\un)| + \int_{E\cap\{|\un| \geq k\}}|\nabla \un|
$$
we have, by \rife{cc}, and by \rife{sceltak},
$$
\ie|\nabla \un| \leq \ie|\nabla T_{k}(\un)| + C \eps\,.
$$
Using H\"older's inequality and \rife{dd}, we obtain
$$
\ie|\nabla \un|
\leq
C\,{\rm meas}(E)^{\frac 12}k^{\frac12}\,(1 + k)^{\frac{\gamma}{2}}\bigg(\io |f|\bigg)^{\frac 12}
+
C \eps\,.
$$
Choosing ${\rm{meas}}(E)$ small enough (recall that $k$ is now fixed) we have
$$
\ie |\nabla \un|\leq C\eps\,,
$$
uniformly with respect to $n$, where $C$ does not depend on $n$ or $\eps$. Since $\nabla \un$ almost everywhere converges to $\nabla u$ by Corollary \ref{duqo}, we can apply Vitali's theorem to obtain the strong convergence of $\nabla \un$ to $\nabla u$ in $(\elle1)^{N}$.

As for the second convergence, by \rife{aa} we have
$$
\arrstre
\ba{r@{\hspace{2pt}}c@{\hspace{2pt}}l}
\disp
\int_{E}|\un|^{\frac{\gamma+2}{2}}
& \leq &
\disp
\int_{E\cap\{|\un| \leq k\}}|\un|^{\frac{\gamma+2}{2}}
+\int_{E\cap\{ |\un| \geq k\}}|\un|^{\frac{\gamma+2}{2}}
\\
& \leq &
\disp
k^{\frac{\gamma+2}{2}}{\rm meas}(E) + \ik|f|^{\frac{\gamma+2}{2}}\,.
\ea
$$
As before, we first choose $k$ such that the second integral is small, uniformly with respect to $n$, and then the measure of $E$ small enough such that the first term is small.
The almost everywhere convergence of $\un$ to $u$, and Vitali's theorem, then imply that $\un$ strongly converges to $u$ in $L^{\frac{\gamma+2}{2}}(\Omega)$.
\end{proof}

\begin{ohss}\label{notapprox}\rm
Since we have proved that $\nabla\un$ strongly converges to $\nabla u$ in $(\elle1)^{N}$, so that $u$ belongs to $\sob{1,1}{0}$, then the approximate gradient $\nabla u$ of $u$ is nothing but the distributional gradient of $u$ (see \cite{BBGGPV}).
\end{ohss}

\begin{proof}[Proof of Theorem \ref{teoroma}] Using the previous results, we pass to the limit, as $n$ tends to infinity, in the weak formulation of \rife{ppn_0}. Starting from
$$
\io \frac{a(x,\nabla\un) \cdot \nabla\vp}{(1 + |\un|)^{\gamma}}
+
\io \un\,\vp
=
\io \fn\,\vp\,,
\qquad
\vp \in \sob{1,\infty}{0}\,,
$$
the limit of the second and the last integral is easy to compute; indeed, recall that by Lemma \ref{dul1}, and by definition of $\fn$, the sequences $\{\un\}$ and $\{\fn\}$ strongly converge to $u$ and $f$ respectively in $\elle{\frac{\gamma+2}{2}}$, hence in $\elle1$. For the first integral, we have that $a(x,\nabla\un)$ converges almost everywhere in $\Omega$ to $a(x,\nabla u)$ thanks to Corollary \ref{duqo}, and the continuity assumption on $a(x,\cdot)$; furthermore, \rife{bdd} implies that
$$
|a(x,\nabla\un)| \leq \beta|\nabla \un|\,,
$$
and the right hand side is compact in $\elle1$ by Lemma \ref{dul1}. Thus, by Vitali's theorem $a(x,\nabla\un)$ strongly converges to $a(x,\nabla u)$ in $(\elle1)^{N}$, so that
$$
\lim_{n \to +\infty}\,\io \frac{a(x,\nabla\un) \cdot \nabla\vp}{(1 + |\un|)^{\gamma}}
=
\io \frac{a(x,\nabla u) \cdot \nabla\vp}{(1 + |u|)^{\gamma}}\,,
$$
where we have also used that $\un$ almost everywhere converges to $u$, and Lebesgue's theorem. Thus, we have that
$$
\io \frac{a(x,\nabla u) \cdot \nabla\vp}{(1 + |u|)^{\gamma}}
+
\io u\,\vp
=
\io f\,\vp\,,
\qquad
\forall \vp \in \sob{1,\infty}{0}\,,
$$
i.e., $u$ satisfies \rife{eq-distr}.
\end{proof}

\section{Uniqueness of the solution obtained by approximation}

Let $f \in \elle{\frac{\gamma+2}{2}}$, let $\fn$ be a sequence of $\elle\infty$ functions converging to $f$ in $\elle{\frac{\gamma+2}{2}}$, with $|\fn| \leq |f|$, and let $\un$ be a solution of \rife{ppn_0}.
In Section~2 we proved the existence of a distributional solution $u$ in $W_0^{1,1}(\Omega)\cap\elle{\frac{\gamma+2}{2}}$
to \rife{P}, such that, up to a subsequence,
\be
\label{gise}
\lim_{n \to +\infty}\,\norma{\un-u}{W_0^{1,1}(\Omega )\cap\elle{\frac{\gamma+2}{2}}} = 0\,.
\ee
Now, let $g \in \elle{\frac{\gamma+2}{2}}$, let $g_{n}$ be a sequence of $\elle\infty$ functions converging to $g$ in $\elle{\frac{\gamma+2}{2}}$, with $|g_{n}| \leq |g|$, and let $\zn$ in $\huz \cap \elle\infty$ be a weak solution of
\be\label{gn_0}
\left\{
\arrstre
\ba{cl}
\disp
-\dive\bigg(\frac{a(x,\nabla\zn)}{(1+|\zn|)^\gamma}\bigg) + \zn = g_n & \mbox{in $\Omega$,} \\
\hfill \zn = 0 \hfill & \mbox{on $\partial\Omega $.}
\ea
\right.
\ee
Then, up to a subsequence, we can assume that
\be
\label{lla}
\lim_{n \to +\infty}\,\norma{\zn-z}{W_0^{1,1}(\Omega )\cap\elle{\frac{\gamma+2}{2}}} = 0\,,
\ee
where $z$ in $W_0^{1,1}(\Omega) \cap \elle{\frac{\gamma+2}{2}}$ is a distributional solution of
\be
\label{Pz}
\left\{
\arrstre
\ba{cl}
\disp
-\dive\bigg(\frac{a(x,\nabla z)}{(1+|z|)^\gamma}\bigg) + z = f & \mbox{in $\Omega $,} \\
\hfill z = 0 \hfill & \mbox{on $\partial\Omega $.}
\ea
\right.
\ee
Our result, which will imply the uniqueness of the solution by approximation (see \cite{d}) of \rife{P}, is the following.

\begin{theo}\label{l1}\sl
Assume that $\un$ and $\zn$ are solutions of \rife{ppn_0} and \rife{gn_0} respectively, and that \rife{gise} and \rife{lla} hold true, with $u$ and $z$ solutions of \rife{P} and \rife{Pz} respectively. Then
\be
\label{stima-l1}
\io|u-z|
\leq
\io|f-g|\,.
\ee
Moreover,
\be
\label{ord}
f\leq g \mbox{ a.e.\ in $\Omega$\quad implies\quad}
u\leq z \mbox{ a.e.\ in $\Omega$.}
\ee
\end{theo}

\begin{proof}
Substracting \rife{gn_0} from \rife{ppn_0} we obtain
$$
-\dive\bigg(\bigg[
\frac{a(x,\nabla\un)}{(1+|\un|)^\gamma}
-\frac{a(x,\nabla\zn)}{(1+|\zn|)^\gamma}
\bigg]\bigg)
+\un-\zn= \fn-g_n\,.
$$
Choosing $T_h(\un-\zn)$ as a test function we have
$$
\arrstre
\ba{l}
\disp
\io \bigg[
\frac{a(x,\nabla\un)}{(1+|\un|)^\gamma}
-\frac{a(x,\nabla\zn)}{(1+|\zn|)^\gamma}
\bigg]
\cdot \nabla T_h(\un-\zn)
\\
\quad
\disp
+
\io(\un-\zn)T_h(\un-\zn)
=\io(\fn-g_n)T_h(\un-\zn)\,.
\ea
$$
This equality can be written in an equivalent way as
$$
\arrstre
\ba{l}
\disp
\io \frac{[a(x,\nabla\un) - a(x,\nabla\zn)] \cdot \nabla T_h(\un-\zn)}{(1+|\un|)^\gamma}
\\
\quad
\disp
+
\io(\un-\zn)T_h(\un-\zn)
=
\io(\fn-g_n)T_h(\un-\zn)
\\
\quad
\disp
-
\io \bigg[
\frac{1}{(1+|\un|)^\gamma}
-\frac{1}{(1+|\zn|)^\gamma}
\bigg]
a(x,\nabla\zn) \cdot \nabla T_h(\un-\zn)\,.
\ea
$$
By \rife{monot}, the first term of the left hand side is nonnegative, so that it can be dropped; using Lagrange's theorem on the last term of the right hand side, we therefore have, since the absolute value of the derivative of the function $s \mapsto \frac{1}{(1 + |s|)^{\gamma}}$ is bounded by $\gamma$,
$$
\arrstre
\ba{l}
\disp
\io(\un-\zn)T_h(\un-\zn)
\leq
\io(\fn-g_n)T_h(\un-\zn)
\\
\quad
\disp
+\gamma h
\io |a(x,\nabla\zn)||\nabla T_h(\un-\zn)|\,.
\ea
$$
Dividing by $h$ we obtain
$$
\arrstre
\ba{l}
\disp
\io(\un-\zn)\frac{T_h(\un-\zn)}{h}
\leq
\io|\fn-g_n|\frac{|T_h(\un-\zn)|}{h}
\\
\quad
\disp
+\gamma
\io |a(x,\nabla\zn)||\nabla T_h(\un-\zn)|\,.
\ea
$$
Since, for every fixed $n$, $\un$ and $\zn$ belong to $\huz$, and $a(x,\xi)$ satisfies \rife{bdd}, the limit as $h$ tends to zero gives
\be\label{uniche}
\disp
\io|\un-\zn|
\leq
\io|\fn-g_n|\,,
\ee
which then yields \rife{stima-l1} passing to the limit and using the second part of Lemma \ref{dul1}.

The use of $T_h(\un-\zn)^+$ as a test function and the same technique as above imply that
$$
\io(\un-\zn)^+
\leq
\int_{\{\un\geq\zn\}}(\fn-g_n)\,.
$$
Hence, passing to the limit as $n$ tends to infinity, we obtain, if we suppose that $f \leq g$ almost everywhere in $\Omega$,
$$
\io(u-z)^+
\leq
\int_{\{ u\geq z\}}(f-g)
\leq 0\,,
$$
so that \rife{ord} is proved.
\end{proof}

Thanks to \rife{stima-l1}, we can prove that problem \rife{P} has a unique solution obtained by approximation.

\begin{cor}\label{unic}\sl
There exists a unique solution obtained by approximation of \rife{P}, in the sense that the solution $u$ in $\sob{1,1}{0} \cap \elle{\frac{\gamma+2}{2}}$ obtained as limit of the sequence $\un$ of solutions of \rife{ppn_0} does not depend on the sequence $\fn$ chosen to approximate the datum $f$ in $\elle{\frac{\gamma+2}{2}}$.
\end{cor}

\begin{ohss}\label{ancheun}\rm
Note that \rife{uniche} implies the uniqueness of the solution of \rife{pinfty}, while \rife{ord} implies that if $f \geq 0$, then the solution $u$ of \rife{P} is nonnegative.
\end{ohss}

\begin{ohss}\label{mazon}\rm
Corollary \ref{unic}, together with estimates \rife{stima-l1} and \rife{aa}, implies that
the map $S$ from $\elle{\frac{\gamma+2}{2}}$ into itself defined by $S(f) = u$, where $u$ is the solution of \rife{P} with datum $f$, is well defined and satisfies
$$
\norma{S(f)-S(g)}{\elle1}\leq\norma{f-g}{\elle1}\,,
\quad
\norma{S(f)}{\elle{\frac{\gamma+2}{2}}}\leq\norma{f}{\elle{\frac{\gamma+2}{2}}}\,.
$$
\end{ohss}

\section{A non existence result}

As stated in the Introduction, we prove here a non existence result for solutions of \rife{P} if the datum is a bounded Radon measure concentrated on a set $E$ of zero harmonic capacity.

\begin{theo}\sl
Assume that $\gamma>1$, and let $\mu$ be a nonnegative Radon measure, concentrated on a set $E$ of zero harmonic capacity. Then there is no solution to
$$
\left\{ \arrstre
\ba{cl}
\disp
-\dive\bigg(\frac{a(x,\nabla u)}{(1+u)^\gamma}\bigg)
+u= \mu& \mbox{in $\Omega $,}\\
u=0 & \mbox{on $\partial\Omega $.}
\ea
\right.
$$
More precisely, if $\{\fn\}$ is a sequence of nonnegative $L^{\infty}(\Omega)$ functions which converges to $\mu$
in the tight sense of measures, and if $\un$ is the sequence of solutions to \rife{ppn_0}, then
$\un$ tends to zero almost everyhwere in $\Omega$ and
$$
\lim_{n\to +\infty} \io \un\,\vp = \io \vp \,d\mu\qquad \forall\, \vp \in \sob{1,\infty}{0}\,.
$$
\end{theo}

\begin{ohss}\rm
A similar non existence result for the case $\gamma\leq 1$ is much more complicated to obtain.
Indeed, if for example $a(x,\xi) = \xi$, and $\gamma = 1$, the change of variables $v=\log (1+u)$ yields that $v$ is a solution to
$$
\left\{ \arrstre
\ba{cl}
\disp
-\Delta v +{\rm e}^{v}-1= \mu& \mbox{in $\Omega $,} \\
u=0 & \mbox{on $\partial\Omega $.}
\ea
\right.
$$
Existence and non existence of solutions for such a problem has been studied in \cite{BMP2} (where the concept of ``good measure'' was introduced) and in \cite{V} (if $N = 2$) and \cite{BLOP} (if $N \geq 3$).
\end{ohss}

\begin{proof}
Let $\mu$ be as in the statement. Then (see \cite{DMOP}) for every $\delta > 0$ there exists a function $\psi_{\delta}$ in $C^{\infty}_{0}(\Omega)$ such that
$$
0 \leq \psi_{\delta} \leq 1\,,
\qquad
\io |\nabla\psi_{\delta}|^{2} \leq \delta\,,
\qquad
\io\,(1 - \psi_{\delta})d\mu \leq \delta\,.
$$
Note that, as a consequence of the estimate on $\psi_{\delta}$ in $\huz$, and of the fact that $0 \leq \psi_{\delta} \leq 1$, $\psi_{\delta}$ tends to zero in the weak$^{*}$ topology of $\elle{\infty}$ as $\delta$ tends to zero.

If $\fn$ is a sequence of nonnegative functions which converges to $\mu$ in the tight convergence of measures, that is, if
$$
\lim_{n \to +\infty}\,\io \fn \, \vp = \io \vp\,d\mu\,,
\qquad
\forall \vp \in C^0(\overline{\Omega})\,,
$$
then
\be\label{azero}
0 \leq \lim_{n \to +\infty}\io\,\fn\,(1 - \psi_{\delta}) = \io \, (1 - \psi_{\delta})\,d\mu \leq \delta\,.
\ee
Let $\un$ be the nonnegative solution to the approximating problem \rife{ppn_0}.
If we choose $1-(1+\un)^{1-\gamma}$ as a test function in \rife{ppn_0}, we have, by \rife{ell}, and dropping the nonnegative lower order term,
$$
\alpha(\gamma-1)\io \bigg|\frac{\nabla \un}{(1+\un)^{\gamma}}\bigg|^2
\leq
(\gamma-1)\io \frac{a(x,\nabla \un) \cdot \nabla \un}{(1+\un)^{2\gamma}}\leq \io \fn\,.
$$
Recalling \rife{bdd}, we thus have
$$
\io \bigg| \frac{a(x,\nabla \un)}{(1+\un)^{\gamma}}\bigg|^{2}
\leq
\beta \io \bigg|\frac{\nabla \un}{(1+\un)^{\gamma}}\bigg|^2
\leq
C \io \fn\,,
$$
with $C$ depending on $\alpha$, $\beta$ and $\gamma$. Therefore, up to a subsequence, there exist $\sigma$ in $(\elle2)^{N}$ and $\rho$ in $\elle2$ such that
\be\label{weak_convergence_dirac}
\lim_{n \to +\infty}\,\frac{a(x,\nabla \un)}{(1+\un)^{\gamma}} = \sigma\,,
\quad
\lim_{n \to +\infty}\,\bigg|\frac{\nabla \un}{(1+\un)^{\gamma}}\bigg| = \rho\,,
\ee
weakly in $(L^2(\Omega))^N$ and $\elle2$ respectively.

The choice of $[1-(1+\un)^{1-\gamma}](1-\psi_{\delta})$ as a test function in \rife{ppn_0} gives
\be\label{conpside}
\arrstre
\ba{l}
\disp
(\gamma-1)\io \frac{a(x,\nabla \un) \cdot \nabla \un}{(1+\un)^{2\gamma}} (1-\psi_{\delta})
\\
\disp
\qquad\quad
+
\io \un [1-(1+\un)^{1-\gamma}](1-\psi_{\delta})
\\
\disp
\qquad
=
\io \fn [1-(1+\un)^{1-\gamma}](1-\psi_{\delta})
\\
\disp
\qquad\quad
+
\io \frac{a(x,\nabla \un) \cdot \nabla \psi_{\delta}}{(1+\un)^{\gamma}}[1-(1+\un)^{1-\gamma}]
\\
\disp
\qquad
\leq
\io \fn(1-\psi_{\delta})
\\
\disp
\qquad\quad
+
\io \frac{a(x,\nabla \un) \cdot \nabla \psi_{\delta}}{(1+\un)^{\gamma}}[1-(1+\un)^{1-\gamma}]\,.
\ea
\ee
We study the right hand side.
For the first term, \rife{azero} implies that
$$
\lim_{\delta \to 0^{+}}\,\lim_{n \to +\infty}\, \io \, \fn\,(1-\psi_{\delta}) = 0\,,
$$
while for the second one, we have, using \rife{weak_convergence_dirac}, and the boundedness of $\disp [1-(1+\un)^{1-\gamma}]$,
$$
\lim_{n \to +\infty}\,
\io \frac{a(x,\nabla \un) \cdot \nabla \psi_{\delta}}{(1+\un)^{\gamma}}[1-(1+\un)^{1-\gamma}]
=
\io \sigma \cdot \nabla \psi_{\delta}[1-(1+u)^{1-\gamma}]\,.
$$
Recalling that $\sigma$ is in $(\elle2)^{N}$, that $\psi_{\delta}$ tends to zero in $\huz$, and using the boundedness $[1-(1+u)^{1-\gamma}]$, we have
$$
\lim_{\delta \to 0^{+}}\,\lim_{n \to +\infty}\,
\io \frac{a(x,\nabla \un) \cdot \nabla \psi_{\delta}}{(1+\un)^{\gamma}}[1-(1+\un)^{1-\gamma}]
=
0\,.
$$
Therefore, since both terms of the left hand side of \rife{conpside} are nonnegative, we obtain
$$
\lim_{\delta \to 0^{+}}\,\lim_{n \to +\infty}\,
\io \frac{a(x,\nabla \un) \cdot \nabla \un}{(1+\un)^{2\gamma}} (1-\psi_{\delta})
= 0\,.
$$
Assumption \rife{ell} then gives
$$
\arrstre
\ba{l}
\disp
\lim_{\delta \to 0^{+}}\,\lim_{n \to +\infty}\,
\alpha \io \bigg|\frac{\nabla \un}{(1+\un)^{\gamma}}\bigg|^2(1-\psi_{\delta})
\\
\disp
\qquad
\leq
\lim_{\delta \to 0^{+}}\,\lim_{n \to +\infty}\,
\io \frac{a(x,\nabla \un) \cdot \nabla \un}{(1+\un)^{2\gamma}} (1-\psi_{\delta})
=
0\,.
\ea
$$
Since the functional
$$
v \in \elle2 \mapsto \io |v|^{2}(1-\psi_{\delta})
$$
is weakly lower semicontinuous on $\elle2$, we have
$$
\io |\rho|^{2}
=
\lim_{\delta \to 0^{+}} \io | \rho |^{2}(1-\psi_{\delta})
\leq
\lim_{\delta \to 0^{+}}\,
\lim_{n \to +\infty}
\io \bigg| \frac{\nabla \un}{(1+\un)^{\gamma}}\bigg|^{2}(1-\psi_{\delta})
=
0\,,
$$
which implies that $\rho = 0$. Thus, since
$$
\frac{\nabla \un}{(1+\un)^{\gamma}}
=
\frac{1}{\gamma-1}\nabla \bigg(1 - (1 + \un)^{1-\gamma}\bigg)\,,
$$
by the second limit of \rife{weak_convergence_dirac} the sequence $1-(1+\un)^{1-\gamma}$ weakly converges to zero in $\huz$, and so (up to subsequences) it strongly converges to zero in $L^2(\Omega)$. Therefore $\un$ (up to subsequences) tends to zero almost everywhere in $\Omega$. Since the limit does not depend on the subsequence, the whole sequence $\un$ tends to zero almost everywhere in $\Omega$.

We now have, for $\Phi$ in $(\elle2)^{N}$, and by \rife{bdd},
$$
\bigg| \io \frac{a(x,\nabla\un)}{(1 + |\un|)^{\gamma}} \cdot \Phi \bigg|
\leq
\io \bigg| \frac{a(x,\nabla\un)}{(1 + |\un|)^{\gamma}} \bigg| |\Phi|
\leq
\beta \io \frac{|\nabla\un|}{(1 + |\un|)^{\gamma}} |\Phi|\,.
$$
Thus, by \rife{weak_convergence_dirac},
$$
\bigg| \io\, \sigma \cdot \Phi \bigg|
=
\lim_{n \to +\infty}\,\bigg| \io \frac{a(x,\nabla\un)}{(1 + |\un|)^{\gamma}} \cdot \Phi \bigg|
\leq
\beta \io \rho\,|\Phi| = 0\,,
$$
which implies that $\sigma = 0$. Therefore, passing to the limit in \rife{ppn_0}, that is, in
$$
\io \frac{a(x,\nabla \un)\cdot \nabla \vp}{(1+\un)^{\gamma}}+\io \un\,\vp = \io \fn \,\vp\,,
\qquad
\vp \in \sob{1,\infty}{0}\,,
$$
we get, since the first term tends to zero,
$$
\lim_{n \to +\infty}\int_{\Omega}\,\un\,\vp = \io\,\vp\,d\mu\,,
$$
for every $\vp$ in $\sob{1,\infty}{0}$, as desired.
\end{proof}

\begin{ohss}\rm
With minor technical changes (see \cite{DMOP}) one can prove the same result if $\mu$ is a signed Radon measure concentrated on a set $E$ of zero harmonic capacity.
\end{ohss}


\end{document}